\documentclass[letterpaper, 10pt, conference]{ieeeconf}    
\IEEEoverridecommandlockouts                       
\usepackage[colorlinks=true,citecolor=blue,linkcolor=blue,urlcolor=blue]{hyperref}
\usepackage{url}           
\usepackage{booktabs}      
\usepackage{nicefrac}      
\usepackage{microtype}     
\usepackage{xcolor}       
\usepackage{tgtermes}
\usepackage{algorithm}
\usepackage{algorithmic}

\usepackage{amsfonts,amsmath,amssymb}
\usepackage{bm}
\usepackage{here}
\usepackage{color}
\usepackage{graphicx}
\usepackage{mathrsfs}
\usepackage[subrefformat=parens]{subcaption}
\usepackage{wrapfig}
\usepackage{tikz}
\usetikzlibrary{positioning}
\usepackage{tabularx}
\usepackage{cite}

\usepackage{amsthm}

\newcommand{\calN}{{\mathcal N}}    
    
\newcommand{\calP}{{\mathcal P}}

\newcommand{\calS}{{\mathcal S}}    
\newcommand{\calT}{{\mathcal T}}

\newcommand{\scrT}{{\mathscr T}}

\newcommand{\bbE}{{\mathbb E}}

\newcommand{\bbN}{{\mathbb N}}

\newcommand{\bbR}{{\mathbb R}}

\newcommand{\rmd}{{\rm d}}

\newcommand{\Ee}{{\rm e}}   %% Euler's 

\newcommand{\ft}{T}

\newcommand{\trace}{{\rm tr}}

\newcommand{\wtilde}[1]{\widetilde{#1}}

\newtheorem{proposition}{Proposition}

\newtheorem{lemma}{Lemma}
\newtheorem{theorem}{Theorem}

\newtheorem{remark}{Remark}

\title{\LARGE \bf
Sliced Wasserstein Steering between Gaussian Measures
}

\author{Kaito Ito$^{1}$ and Anqi Dong$^{2}$% <-this % stops a space
\thanks{*This work was supported in part by JSPS KAKENHI Grant Number JP24K17297, JST, Adopting Sustainable Partnerships for Innovative Research Ecosystem (ASPIRE), Grant Number JPMJAP2402, and Swedish Research Council Distinguished Professor Grant 2017-01078, and Knut and Alice Wallenberg Foundation Wallenberg Scholar Grant.}% <-this % stops a space
\thanks{$^{1}$Kaito Ito is with the Department of Information Physics and Computing, The University of Tokyo, Tokyo 113-8654, Japan {\tt \small kaito@g.ecc.u-tokyo.ac.jp}}%
\thanks{$^{2}$Anqi Dong is with the Department of Decision and Control Systems and Department of Mathematics, KTH Royal Institute of Technology, Stockholm 114~28, Sweden {\tt\small anqid@kth.se}}%
\thanks{\copyright 2026 IEEE. Personal use of this material is permitted. Permission from IEEE must be obtained for all other uses, in any current or future media, including reprinting/republishing this material for advertising or promotional purposes, creating new collective works, for resale or redistribution to servers or lists, or reuse of any copyrighted component of this work in other works.}
}

\begin{document}

\maketitle
\thispagestyle{empty}
\pagestyle{empty}

\begin{abstract}
Optimal transport with quadratic cost provides a geometric framework for steering an ensemble, modeled by a probability law, with minimal effort. Yet ambient-space formulations become unwieldy in high dimensions, and sensing or actuation in practice often reveals only linear views of the state --- camera silhouettes, LiDAR beams, tomographic slices. We develop a sliced feedback controller for distribution steering: the evolving law is projected onto one-dimensional directions on the sphere, the optimal one-dimensional velocity is synthesized in each projection, and these velocities are averaged to produce a feedback control in the ambient space. 
The construction reduces to the Benamou--Brenier problem in one dimension.
In addition, it is invariant under orthogonal transforms, nonexpansive under projections, and well posed on $\mathcal{P}_2(\mathbb{R}^n)$. Computation proceeds by sampling directions on the sphere and solving independent one-dimensional subproblems, yielding a scalable method aligned with partial observations.
In the Gaussian setting, we show that the developed sliced controller steers the law to the prescribed target. Furthermore, we derive an identity relating the energy consumption incurred by the controller to the sliced Wasserstein distance.

\end{abstract}


\section{Introduction}
Optimal transport theory provides a principled way to transform one probability distribution into another with minimal effort, endowing the probability space with the Wasserstein geometry. Originating in Monge’s formulation \cite{Monge1781memoire} and developed through Kantorovich’s relaxation \cite{Kantorovich1942translocation} and a rich geometric analysis, optimal transport now underpins applications across mathematics, economics, and machine learning \cite{Villani2003,Rachev1998mass,Peyre2019computational}. From a control viewpoint, it is natural to regard optimal transport as ensemble steering: rather than moving a single trajectory, one shapes an entire state distribution over a finite horizon while penalizing quadratic effort. The dynamic Benamou--Brenier formulation \cite{Benamou2000} makes this viewpoint explicit by treating the density as the state and the velocity field as the control, with mass conservation providing the dynamics and an action functional measuring effort. For linear dynamics, feasibility aligns with reachability, and the evolution admits structure at the level of moments, placing distribution control on same footing as classical optimal control \cite{Chen2017optimal,chen2018optimal,Chen2021optimal1}. A stochastic counterpart, the Schrödinger bridge \cite{chen2016relation}, offers an entropy-regularized interpolation and recasts quadratic-cost transport in the small-noise limit.

Sliced optimal transport addresses the computational burden that hinders high-dimensional applications in optimal transport. The idea is to replace a single high-dimensional coupling with one-dimensional problems obtained by projection. Concretely, one projects measures along many directions, as in the Radon transform, solves the induced one-dimensional transport or barycenter problems in closed form, and aggregates across directions to obtain sliced distances and barycenters \cite{bonneel2015sliced,paulin2020sliced,Deans2007radon,Peyre2019computational}. This project-then-average construction yields a sample-friendly objective with modest memory footprint and parallel computation, scaling with the number of samples and directions rather than with ambient dimension. Beyond scalability, it also matches how data are acquired in practice: camera silhouettes are line-of-sight projections, LiDAR returns ranges along rays, and CT scans collect line integrals at multiple orientations \cite{Haker2004optimal}. 
These observations motivate our contribution: a sliced optimal transport framework for distribution steering with partial or distributed observations. We thus propose the objective in the output space induced by linear sensing and actuation, and enforce transport along informative projections aggregated across directions. The construction avoids high-dimensional couplings and operates directly on samples, making it compatible with distributed storage and computation. In the full-information limit, it recovers the classical setting. With limited views, it provides a principled surrogate aligned with what sensors and networks actually reveal.

Herein, we develop a sliced optimal transport framework for ensemble steering that operates through one-dimensional projections. The evolving probability law is projected onto directions on the sphere, the optimal one-dimensional velocity is synthesized in each projection, and these velocities are averaged to produce a feedback control in the ambient space. In continuous time, this feedback decreases the sliced Wasserstein objective monotonically with a gain that scales with the inverse of the remaining horizon. For Gaussian marginals, the controller is affine and yields closed equations for mean and covariance. The flow reaches a target distribution at the terminal time, with expected control energy equal to one-half of the squared sliced distance. A discrete-time, direction-by-direction implementation of the sliced distribution steering, called the iterative sliced controller, uses a single one-dimensional map at each time step. It agrees numerically with the ideal law obtained by using infinitely many slicing directions, while differing qualitatively from the minimum-energy displacement interpolation. Lastly, we note that, while previous studies~\cite{Vauthier2025,Liutkus2019,Cozzi2025} have investigated gradient flows for the sliced optimal transport describing asymptotic evolution, the present work is the first to analyze finite-time sliced distribution steering.

The organization of the paper is as follows. Preliminaries on Wasserstein distance with its sliced variant and notations appear in Section~\ref{sec:prelim}. Section~\ref{sec:iterative} introduces the iterative projected controller in discrete time and its MPC-style application of the first control. Section~\ref{sec:single_integrator} establishes the time derivative of the squared sliced distance $SW_2^2$ (\textbf{Proposition}~\ref{prop:time_derivative}) and develops the continuous-time sliced controller for the single integrator so that $SW_2$ decreases. Then, we derive its affine form for Gaussian marginals (\textbf{Proposition}~\ref{prop:linearlity_Gaussian}), and prove convergence and an energy identity (\textbf{Theorem}~\ref{thm:convergence} and~\ref{thm:energy}). Section~\ref{sec:num} presents numerical comparisons between the iterative sliced controller, the continuous-time sliced feedback, and the minimum-energy controller. Section~\ref{sec:conclusion} concludes and outlines extensions to general linear plants and sensing-limited settings.

\section{Preliminaries}\label{sec:prelim}

Let $\calP_2(\bbR^n)$ denote the set of Borel probability measures on $\bbR^n$ with finite second moment, and $C_c^\infty(S)$ denote the space of real-valued, compactly supported, smooth functions on a set $S$. The unit sphere in $\bbR^n$ is $\calS^{n-1}$. When clear from context, we use the same symbol for a probability measure and its density with a slight abuse of notation. Furthermore, we write $\partial_t f$ for $\partial f/\partial t$, and $\nabla_x$ for the gradient with respect to $x$. The Gaussian law with mean $\mu$ and covariance $\Sigma$ is $\calN(\mu,\Sigma)$, with corresponding density $\calN(\,\cdot\,\mid \mu,\Sigma)$.

For given marginals $\mu,\nu\in\calP_2(\bbR^n)$, the quadratic Wasserstein distance is defined through the Monge--Kantorovich optimal transport problem \cite{Villani2003,Rachev1998mass,dong2024monge}, i.e.,
\begin{equation}\label{def:wasserstein}
W_2(\mu,\nu)^2
:= \inf_{\gamma\in\Gamma(\mu,\nu)} \int_{\bbR^n\times\bbR^n} \|x-y\|^2 \,\gamma(\rmd x,\rmd y),
\end{equation}
over the set of all admissible couplings 
\begin{align*}
\Gamma(\mu,\nu)
:= \Bigl\{&\gamma\in\calP(\bbR^n\times\bbR^n)\ \Bigm|\ \gamma(X_1\times\bbR^n)=\mu(X_1),\\
&\gamma(\bbR^n\times X_2)=\nu(X_2)\ \text{for all Borel }X_1,X_2\Bigr\}.    
\end{align*}

In one-dimensional settings, if $F_\mu$ and $F_\nu$ are the cumulative distribution functions of $\mu$ and $\nu$, and if $F_\mu^{-1}(z)=\inf\{s\in\bbR\mid F_\mu(s)\ge z\}$ denotes the left-continuous generalized inverse~\cite[Theorem~2.18]{Villani2003}, then
\[
W_2(\mu,\nu)^2
= \int_0^1 \bigl(F_\mu^{-1}(z)-F_\nu^{-1}(z)\bigr)^2 \,\rmd z.
\]

For $\theta\in\calS^{n-1}$, let $P_\theta:\bbR^n\to\bbR$ be the projection $P_\theta(x)=\theta^\top x$. The push-forward of $\mu$ by $P_\theta$ is denoted as $P_\theta\#\mu$, defined by $P_\theta\#\mu(X)=\mu(P_\theta^{-1}(X))$ for every Borel set $X\subset\bbR$. The sliced quadratic Wasserstein distance between $\mu,\nu\in\calP_2(\bbR^n)$ is defined as 
\begin{equation}\label{def:slice_wasserstein}
SW_2(\mu,\nu)^2
:= \int_{\calS^{n-1}} W_2\bigl(P_\theta\#\mu,\ P_\theta\#\nu\bigr)^2 \,\sigma(\rmd\theta),
\end{equation}
where $\sigma$ is the uniform probability measure on $\calS^{n-1}$. In words, $SW_2$ averages the one-dimensional transport between the linear projections of $\mu$ and $\nu$. In the one-dimensional case, $W_2$ admits a closed-form in terms of cumulative distribution functions and their inverses, making the sliced construction straightforward to compute.

The control viewpoint of the quadratic cost goes back to Benamou--Brenier \cite{Benamou2000}, and later \cite{Chen2021optimal1}. In Eulerian form, one steers a density by a velocity field and minimizes quadratic action:
\begin{align}\label{eq:bb-eulerian}
\begin{aligned}
&\inf_{\rho,v}\ &&\int_0^{\ft}\!\int_{\bbR^n} \frac12\,\rho(t,x)\,\|v(t,x)\|^2 \,\rmd x\,\rmd t\\
&~\text{s.t.}
&&\partial_t \rho + \nabla_x\!\cdot(\rho v)=0,\\
& &&\rho(0,\cdot)=\mu,\ \ \rho(\ft,\cdot)=\nu.    
\end{aligned}
\end{align}
In Lagrangian form, one minimizes expected control energy along trajectories as
\begin{align}\label{eq:bb-lagrangian}
\begin{aligned}
&\inf_{u(\cdot)}\ &&\int_0^{\ft} \mathbb E\!\left[\frac{1}{2} \|u(t)\|^2\right]\rmd t\\
&~\text{s.t.}\quad &&\dot x(t)=u(t),\\
& &&x(0)\sim\mu,\ \ x(\ft)\sim\nu.   
\end{aligned}
\end{align}
Notably, the two formulations in \eqref{eq:bb-eulerian} and \eqref{eq:bb-lagrangian} are in fact equivalent. This correspondence is not limited to the classical setting, but continues to hold for linear dynamics of the form $\dot x=A(t)x+B(t)u(t)$ under standard reachability conditions, namely, absolute continuity of the Brenier map and controllability of the pair $(A,B)$ on $[0,\ft]$, see e.g.,~\cite{Chen2017optimal}.

\section{Iterative Sliced Optimal Transport}\label{sec:iterative}
In this paper, we consider the single integrator system
\begin{align}
    \dot{x} (t) &= u(t) , ~~ t \ge 0, \label{eq:integrator}\\
    x(0) &\sim \rho_0 (\cdot)\nonumber ,
\end{align}
where $ x(t) \in \bbR^n $ denotes the state, $ u(t) \in \bbR^n $ denotes the control input, and $ \rho_0 $ denotes the probability density function of the initial state $ x(0) $. 
Under the feedback controller (or the velocity field) $ u(t) = v (t,x(t)) $, the probability law $ \mu_t $ of $ x(t) $ solves the continuity equation $ \partial_t \mu_t = - \nabla_x \cdot ( \mu_t v )  $ in a weak  (distributional)  sense~\cite{Ambrosio2005},~\cite[Chapter 8]{Villani2003}.

For simplicity, we assume throughout that $x(t)$ admits a density $\rho(t,\cdot)$ solving the continuity equation in the classical sense for $(t,x)\in(0,\infty)\times\bbR^n$, namely,
\begin{align}\label{eq:continuity_eq}
\begin{aligned}
\partial_t \rho (t,x) &= - \nabla_x \cdot \left(\rho(t,x) v(t,x) \right),\\
\rho(0,x) &= \rho_0 (x).   
\end{aligned}
\end{align}
Then, for given initial and target densities $ \rho_0 $ and $\rho_f $, we wish to find a controller $ v $ that steers $ \rho $ from $ \rho_0 $ to $ \rho_f $ over a finite horizon $ T > 0 $, i.e., $\rho(T,x) = \rho_f$.

Following \cite{Chen2017optimal}, the minimum-energy density controller that solves \eqref{eq:bb-lagrangian} with $\mu=\rho_0$ and $\nu=\rho_f$ is
\begin{equation}\label{eq:minimum_energy}
v(t,x)=-\frac{1}{T}\big(\calT\circ\calT_{0:t}^{-1}(x)-\calT_{0:t}^{-1}(x)\big),\, (t,x)\in[0,T]\times\bbR^n,
\end{equation}
where $\calT:\bbR^n\to\bbR^n$ is the optimal transport map pushing $\rho_0$ to $\rho_f$ with
\[
W_2(\rho_0,\rho_f)^2=\int_{\bbR^n}\|x-\calT(x)\|^2\,\rho_0(x)\,\rmd x,
\]
and 
$$
\calT_{0:t}(x) :=\frac{T-t}{T}\,x+\frac{t}{T}\,\calT(x)
$$
is the displacement interpolation. However, the computation of the optimal map $ \calT $ between continuous distributions is generally difficult in high dimensions. To this end, we utilize sliced optimal transport.

Before introducing the method, we first discretize the continuous-time system~\eqref{eq:integrator}. Let the sampling instants be $ t_k = kh $ for $ k \in \{0,1,\ldots,T/h\} $, where $ h > 0 $ is a sampling period such that $ T_d := T/h \in \bbN $, $ t_0 = 0 $, and $ t_{T_d} = T $. 
The input is held constant on each interval, $u(t)=u_k$ for $t\in[t_k,t_{k+1})$. For the single-integrator $\dot x=u$, this yields the exact step:
\begin{align}\label{eq:discrete_sys}
x_{k+1}=x_k+h u_k .
\end{align}

Now, we explain our approach, that is inspired by sliced optimal transport and model predictive control (MPC)~\cite{Rawlings2017}. At each time step $ k $, we select a slicing direction $ \theta_k \in \calS^{n-1} $ and compute the one-dimensional optimal transport map $ \scrT_{k}^{\theta_k} : \bbR \rightarrow \bbR $ that moves $ \rho_{\theta_k}(t_k,\cdot) $ to $ \rho_{f,{\theta_k}} $. Here, $ \rho_{\theta_k}(t_k,\cdot) $ and $ \rho_{f,{\theta_k}} $ denote the density functions obtained by the push-forward of $ \rho(t_k,\cdot) $ and $ \rho_f $ through $ P_{\theta_k}  $, respectively, i.e.,
\begin{align*}
    \rho_{\theta_k} (t_k,s) &= \int_{\bbR^n} \delta \! \left(s-\theta_k^\top x\right) \! \rho(t_k,x) \rmd x, && s \in \bbR, \\
    \rho_{f,\theta_k} (s) &= \int_{\bbR^n} \delta \! \left(s-\theta_k^\top x\right) \! \rho_f(x) \rmd x , && s \in \bbR,
\end{align*}
where $\delta$ denotes the Dirac delta function.
Next, we calculate a controller that steers the density $ \rho(t_k,\cdot) $ to the target $ \rho_f $ along the direction $ \theta_k $. By projecting the state $ x_k $ onto the direction $ \theta_k $, i.e., $ s_k = \theta_k^\top x_k \in \bbR $, the system \eqref{eq:discrete_sys} is transformed into 
$$
s_{k+1} = s_k + h u_{\theta_k,k},
$$
where $ u_{\theta_k,k} := \theta_k^\top u_k \in \bbR $. Since the mass of $\rho_{\theta_k} (t_k,s)$ at $ s = \theta_k^\top x $ should be moved to $ \scrT_k^{\theta_k} (\theta_k^\top x) $ for any $x\in \bbR^n$, we solve the following fixed end-point optimal control problem
\begin{align}
    &\inf_{\{u_{\theta_k,\ell}\}_{\ell = k}^{T_d -1}} && \sum_{\ell=k}^{T_d-1}  u_{\theta_k,\ell}^2 \label{prob:discrete_end_point}\\
    &\quad~~ {\rm s.t.} ~  && s_{\ell+1} = s_\ell + h u_{\theta_k,\ell}, ~~ \ell \in \{k,\ldots,T_d-1\}, \nonumber \\
    & && s_k = \theta_k^\top x , \ s_{T_d} = \scrT_k^{\theta_k} (\theta_k^\top x). \nonumber
\end{align}
The optimal control is explicitly given as 
$$
u_{\theta_k,\ell}^* = -\frac{ \theta_k^\top x - \scrT_k^{\theta_k} (\theta_k^\top x)}{T - t_k} 
$$ 
for all $ \ell \in \{k,\ldots,T_d-1\}$.\\

Similarly to MPC, only the first control input $ u_{\theta_k,k}^* $ from the optimal input sequence $ \{u_{\theta_k,\ell}^* \}_{\ell=k}^{T_d-1} $ is applied to the system. 
For $ u_{\theta_k,k}^* $ in the projected space, we employ the following control input in the original space satisfying $ \theta_k^\top u_k = u_{\theta_k,k}^* $:
\begin{equation}\label{eq:slice_controller}
    u_k = v_k (x_k;\theta_k) := - \frac{1}{T - t_k} \! \left( \theta_k^\top x_k - \scrT_k^{\theta_k} (\theta_k^\top x_k) \right) \! \theta_k .
\end{equation}

In the next time step, a new slicing direction $ \theta_{k+1} $ is chosen, and the optimal transport problem is solved again. This procedure is repeated until $ k = T_d $, which corresponds to the terminal time $ t = T $.
We call \eqref{eq:slice_controller} the iterative sliced controller.

When the sampling period $ h $ is large, the state density $ \rho $ will fail to reach the target $ \rho_f $ because only a limited number of directions are considered until the terminal time. We expect that as $ h $ becomes small, the state density is steered closer to the target under the proposed controller~\eqref{eq:slice_controller}. To verify this, in the next section, we analyze the idealized setting where infinitely many slicing directions are available at each time, informally corresponding to the limit as $ h \rightarrow 0 $.

\begin{remark}
In the finite-horizon density controller~\eqref{eq:slice_controller}, the denominator $ T - t_k $ is time-dependent. If we replace the horizon $ T_d -1 $ in \eqref{prob:discrete_end_point} by a receding horizon $ T_d + k -1 $ like MPC, the state equation is given by
\begin{align*}
    x_{k+1} &= x_k - \frac{1}{T} \! \left( \theta_k^\top x_k - \scrT_k^{\theta_k} (\theta_k^\top x_k) \right) \! \theta_k .
\end{align*}
This update is equivalent to performing a gradient descent method that moves $ \rho $ towards $ \rho_f $ along the direction $ \theta_k $, while changing $ \theta_k $ at each time step. The coefficient $ T^{-1} $ corresponds to the step size.
Moreover, if, instead of a single direction, we employ a set of slicing directions $ \{\theta_1,\ldots,\theta_n\} $ that forms an orthogonal basis of $\bbR^n$, the resulting update rule coincides with the iterative distribution transfer algorithm proposed in \cite{Pitie2007automated}.
The idea of integrating optimal transport and MPC can also be found in \cite{Ito2023entropic,Ito2023lcss} for discrete measures.
\end{remark}

\section{Sliced Distribution Steering between Gaussians}\label{sec:single_integrator}
In this section, we construct a controller that steers $\rho$ to the target $\rho_f$ based on their sliced Wasserstein distance. The resulting controller can be seen as the limiting case of the iterative sliced controller~\eqref{eq:slice_controller}. Moreover, we reveal a connection between sliced optimal transport and density control via the energy consumption incurred by the constructed controller.

\subsection{Gradient analysis}

The construction of such a density controller is based on the time derivative of $ SW_2 (\rho, \rho_f)^2 $ with respect to \eqref{eq:continuity_eq}.
Denote by $ \rho_\theta(t,\cdot) $ and $ \rho_{f,\theta} $ the density functions of $ P_{\theta}\# \rho(t,\cdot) $ and $ P_{\theta}\# \rho_f $, respectively, which we call sliced densities. Then, the time derivative is given as follows, and the proof can be found in Appendix~\ref{app:time_derivative}.
\begin{proposition}\label{prop:time_derivative}
    Assume that differentiation and integration can be interchanged for almost all $ t\in (0,T) $ as follows:\footnote{A sufficient condition for the interchange of differentiation and integration can be found, for example, in \cite[Theorem~2.27]{Folland1999real}.}
    \begin{align}
        \frac{\rmd }{\rmd t} SW_2 (\rho(t,\cdot), \rho_f)^2 = \int_{\calS^{n-1}}  \partial_t W_2^2 (\rho_\theta (t,\cdot), \rho_{f,\theta} ) \sigma(\rmd \theta) . \label{eq:interchange}
    \end{align}
    Assume also that for any $ i \in \{1,\ldots,n\} $ and $ t \ge 0 $, we have 
    $$ 
    \rho(t,x) v_i (t,x) \rightarrow 0 \ \mbox{ as }\ |x_i| \rightarrow \infty. 
    $$
    Then, for any $ \rho_0, \rho_f \in \calP_2 (\bbR^n) $ and almost all $ t \in (0,T) $, it holds that
\begin{align}
    &\frac{\rmd }{\rmd t} SW_2 (\rho(t,\cdot), \rho_f)^2 \nonumber\\
    &= 2 \int_{\bbR^n} \left[ \int_{\calS^{n-1}} \left(\theta^\top x - \calT_t^\theta (\theta^\top x)  \right) \theta \sigma(\rmd \theta) \right]^\top  \nonumber\\
    &\quad \times v(t,x) \rho(t,x) \rmd x , \label{eq:dSW_dt}
\end{align}
where $ \calT_t^\theta : \bbR \rightarrow \bbR $ is the unique optimal transport map from $ \rho_\theta (t,\cdot) $ to $ \rho_{f,\theta} $, that is,
\begin{equation*}
    W_2 (\rho_{\theta} (t,\cdot), \rho_{f,\theta} )^2 = \int_{\bbR} \left| s - \calT_t^\theta (s) \right|^2 \rho_{\theta} (t,s) \rmd s .
\end{equation*}
\end{proposition}

\subsection{Reachability between Gaussian measures}
In what follows, we assume that the initial and target densities are given by Gaussian distributions, i.e.,
\begin{equation}\label{eq:gaussian_marginals}
    \rho_0(\cdot)  = \calN(\cdot \mid m_0, \Sigma_0), ~~ \rho_f(\cdot)  = \calN( \cdot \mid m_f, \Sigma_f) ,
\end{equation}
where $ \Sigma_0, \Sigma_f \succ 0 $ are positive definite.
The time derivative \eqref{eq:dSW_dt} of $ SW_2 (\rho(t,\cdot),\rho_f)^2 $ motivates the following feedback control law for steering $ \rho(t,\cdot) $ towards $ \rho_f $.
\begin{align}
    v(t,x) := - \lambda(t) \int_{\calS^{n-1}} \left(\theta^\top x - \calT_t^\theta (\theta^\top x) \right) \theta \sigma (\rmd \theta) ,& \nonumber\\
    (t,x) \in [0,T) \times \bbR^n ,& \label{eq:optimal_controller_single_integ}
\end{align}
where $ \lambda(t) := (T-t)^{-1} $.
Under this controller~\eqref{eq:optimal_controller_single_integ}, by Proposition~\ref{prop:time_derivative} and  for almost all $ t\in (0,T) $, we have 
\begin{align}
    &\frac{\rmd }{\rmd t} SW_2 (\rho(t,\cdot), \rho_f)^2 \nonumber\\
    &= -2 \lambda(t) \int_{\bbR^n} \left\| \int_{\calS^{n-1}} \left(\theta^\top x - \calT_t^\theta (\theta^\top x)  \right) \theta \sigma(\rmd \theta) \right\|^2 \nonumber\\
    &\quad \times \rho(t,x) \rmd x \le 0 . \label{eq:dSW_dt_1}
\end{align}
Therefore, $ SW_2 (\rho(t,\cdot), \rho_f) $ is non-increasing.
The proposed controller~\eqref{eq:optimal_controller_single_integ} is also obtained by averaging the iterative sliced controller~\eqref{eq:slice_controller} with respect to the uniform measure $ \sigma $ on $ \calS^{n-1} $. 
Since \eqref{eq:slice_controller} with $\theta_k$ sampled from $\sigma$ is expected to converge to \eqref{eq:optimal_controller_single_integ} as $h \rightarrow 0$, we also refer to \eqref{eq:optimal_controller_single_integ} as the ideal sliced controller.

Under \eqref{eq:gaussian_marginals}, the sliced densities of $ \rho_0 $ and $ \rho_f $ are also Gaussian, i.e., 
\begin{align*}
\rho_\theta (0,y) &= \calN(y \mid \theta^\top m_0, \theta^\top \Sigma_0 \theta),\\
\rho_{f,\theta} (y) &= \calN(y \mid \theta^\top m_f, \theta^\top \Sigma_f \theta).
\end{align*}
By \cite[Remark~2.31]{Peyre2019computational}, for all $\theta\in\calS^{n-1}$, the optimal transport map $\calT_0^\theta$ between $\rho_\theta(0,\cdot)$ and $\rho_{f,\theta}$ is explicitly
\[
\calT_0^\theta(\theta^\top x)
= \theta^\top m_f
  + \sqrt{\frac{\theta^\top \Sigma_f \theta}{\theta^\top \Sigma_0 \theta}}
    \,(\theta^\top x - \theta^\top m_0).
\]
Consequently, the control law \eqref{eq:optimal_controller_single_integ} at initial time reads
\begin{align}
    v(0,x) &= - \frac{\lambda(0)}{n} x + \lambda(0) \int \calT_t^\theta (\theta^\top x) \theta \rmd \sigma \nonumber\\
    &= \lambda(0) \! \left[ \left(\int r(0,\theta) \theta \theta^\top \rmd \sigma \right) - \frac{1}{n} I \right] x \nonumber\\
    &\quad - \lambda(0) \! \left[ \left( \int r(0,\theta) \theta \theta^\top \rmd \sigma  \right) m_0 - \frac{1}{n} m_f  \right] , \label{eq:initial_feedback}
\end{align}
where $\displaystyle r(0,\theta) := \sqrt{\frac{\theta^\top \Sigma_f \theta}{\theta^\top \Sigma_0 \theta}}  $. This is an affine state-feedback controller. Since Gaussianity is preserved under linear dynamics, it is expected that for any $ t \in (0,T) $, the state $ x(t) $ follows a Gaussian distribution and the proposed controller is affine. Indeed, the following result shows the Gaussianity of the state and the affine structure of the proposed controller.
The proof is provided in Appendix~\ref{app:linearity_Gaussian}.
\begin{proposition}\label{prop:linearlity_Gaussian}
Assume that the differential equation $ \dot{x}(t) = v(t,x(t)) $ with \eqref{eq:gaussian_marginals} and \eqref{eq:optimal_controller_single_integ} has a unique solution on $ [0,T) $. Assume also that the following differential equation admits a unique solution satisfying $ \Sigma(t) \succ 0 $ for any $ t \in (0,T] $.
\begin{align*}
\dot{\Sigma} (t) &= K(t,\Sigma(t)) \Sigma(t) +  \Sigma(t) K(t,\Sigma(t))^\top,\\
\Sigma (0) &= \Sigma_0 , \\
K(t,\Sigma) &:= \lambda(t)\Big[ \int_{\calS^{n-1}} \sqrt{\frac{\theta^\top \Sigma_f \theta}{\theta^\top \Sigma \theta}} \theta \theta^\top \sigma(\rmd \theta) - \frac{1}{n} I \Big].
\end{align*}
Then, under the initial and target densities \eqref{eq:gaussian_marginals}, it holds that
\begin{align}
\dot{x}(t) &= K(t,\Sigma(t)) x(t) + \eta(t,\Sigma(t)), \\
x(t) &\sim \calN(m(t),\Sigma(t))
\end{align}
for any $ t \in [0,T) $, and with $\Sigma \succ 0$, we have 
\begin{align}
\eta(t,\Sigma) &:= - \lambda(t)  \Bigl[ \Bigl( \int \sqrt{\frac{\theta^\top \Sigma_f \theta}{\theta^\top \Sigma \theta}} \theta \theta^\top \rmd \sigma  \Bigr) m(t) - \frac{1}{n} m_f  \Bigr] , \nonumber\\
\dot{m}(t) &= - \frac{\lambda(t)}{n} (m(t) - m_f), \nonumber\\
m(0) &= m_0 \nonumber.
\end{align}
\end{proposition}

By the Gaussianity of $ x(t) $, the state density $ \rho(t,\cdot) $ is completely characterized by the mean $ m(t) $ and the covariance $ \Sigma(t) $. The following result shows that the state density is steered to the target $ \rho_f $ under the proposed controller~\eqref{eq:optimal_controller_single_integ}.
The proof is shown in Appendix~\ref{app:convergence}.
\begin{theorem}\label{thm:convergence}
Suppose that the assumptions in Propositions~\ref{prop:time_derivative} and \ref{prop:linearlity_Gaussian} are satisfied. Then, it holds that
\begin{align}
\lim_{t\nearrow T} m(t) = m_f ,~~ \lim_{t\nearrow T} \Sigma(t) = \Sigma_f .
\end{align}
\end{theorem}

Note that even if $\rho_0$ and $\rho_f$ are not Gaussian, the controller 
$$u(t) = K(t,\Sigma(t))x(t) + \eta(t,\Sigma(t))$$ steers the mean $m(t)$ and covariance $\Sigma(t)$ from the initial mean $m_0$ and covariance $\Sigma_0$ to the target mean $m_f$ and covariance $\Sigma_f$. Therefore, the sliced controller can be used for covariance steering.

\subsection{Relationship between sliced optimal transport and density control}
We are now in the position to establish a relationship between the sliced optimal transport and density control. For quadratic-cost optimal transport, minimum-energy density control yields the identity (see \cite{Benamou2000,Chen2017optimal})
\begin{align*}
    W_2 (\rho_0,\rho_f)^2 &= && \inf_{v} \  T \int_{\bbR^n} \int_0^T \| v(t,x) \|^2 \rho(t,x) \rmd x \rmd t \\
    & && \ {\rm s.t.} \ \text{\eqref{eq:continuity_eq}}, \ \rho(0,\cdot) = \rho_0, \ \rho(T,\cdot) = \rho_f .
\end{align*}
By analogy, the sliced Wasserstein objective also admits an energy characterization as shown in Theorem~\ref{thm:energy}, with proof given in Appendix~\ref{app:energy}.
\begin{theorem}\label{thm:energy}
Suppose that the assumptions in Propositions~\ref{prop:time_derivative} and \ref{prop:linearlity_Gaussian} are satisfied. Then, under the control law \eqref{eq:optimal_controller_single_integ}, it holds that
\begin{align*}
\bbE\! \left[\int_0^T \| u(t) \|^2 \rmd t \right] =\frac{1}{2} SW_2 (\rho(0,\cdot),\rho_f)^2 .
\end{align*}
\end{theorem}

\section{Numerical Experiments}\label{sec:num}

We present a comparison between the proposed density controller and the ideal one.
The initial and target densities are given by $ \calN(m_0,\Sigma_0) $ and $ \calN(m_f,\Sigma_f) $, respectively, where
\begin{align*}
m_0 &=
\begin{bmatrix}
 -2\\
 2
\end{bmatrix},
&&\Sigma_0 = \begin{bmatrix}
1 & 0.2 \\ 0.2 & 0.5
\end{bmatrix},\\
m_f &=
\begin{bmatrix}
-8\\
4
\end{bmatrix}, 
&&\Sigma_f = \begin{bmatrix}
0.1 & 0 \\ 0 & 0.04
\end{bmatrix}.
\end{align*}
Fig.~\ref{fig:iterative} shows sample paths of $ x(t) $ under the iterative sliced controller \eqref{eq:slice_controller} with different numbers of discrete time steps $ T_d = T/h = 100, 1000, 100000 $. 
Slicing directions are sampled uniformly on $\calS^{n-1}$.
Note that $ T_d $ equals the number of slices over time horizon $T$.
Even for small $ T_d $, the state density $ \rho $ is steered close to the target $ \rho_f $, and as $ T_d $ increases, the terminal density $ \rho(T,\cdot) $ becomes closer to the target $ \rho_f $.

\begin{figure}[t]
	\begin{minipage}[b]{1.0\linewidth}
		\centering
		\includegraphics[keepaspectratio, scale=0.43]
		{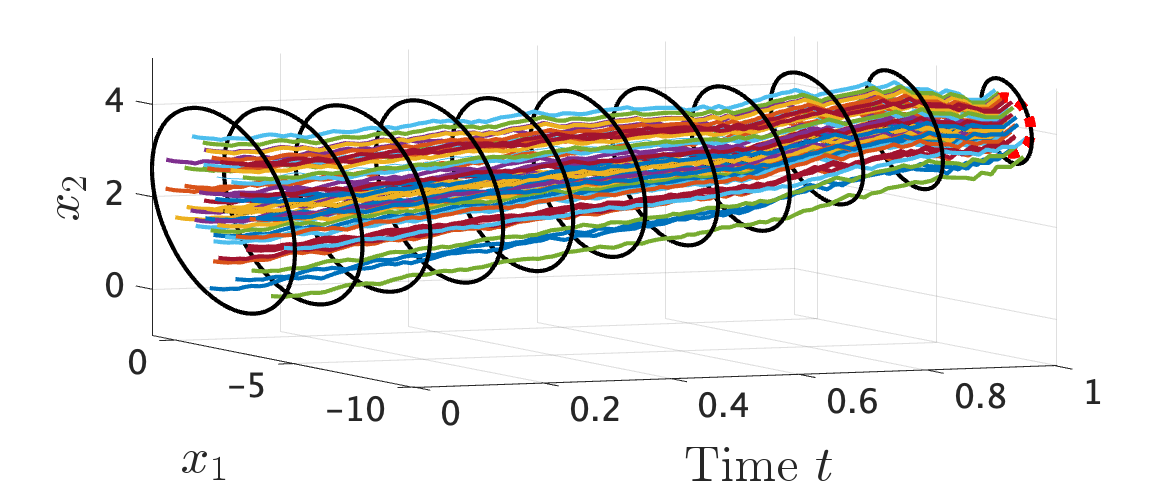}
		\subcaption{$ T_d = 100 $.}\label{fig:opt_state_Q0}
	\end{minipage}
	\begin{minipage}[b]{1.0\linewidth}
		\centering
		\includegraphics[keepaspectratio, scale=0.43]
		{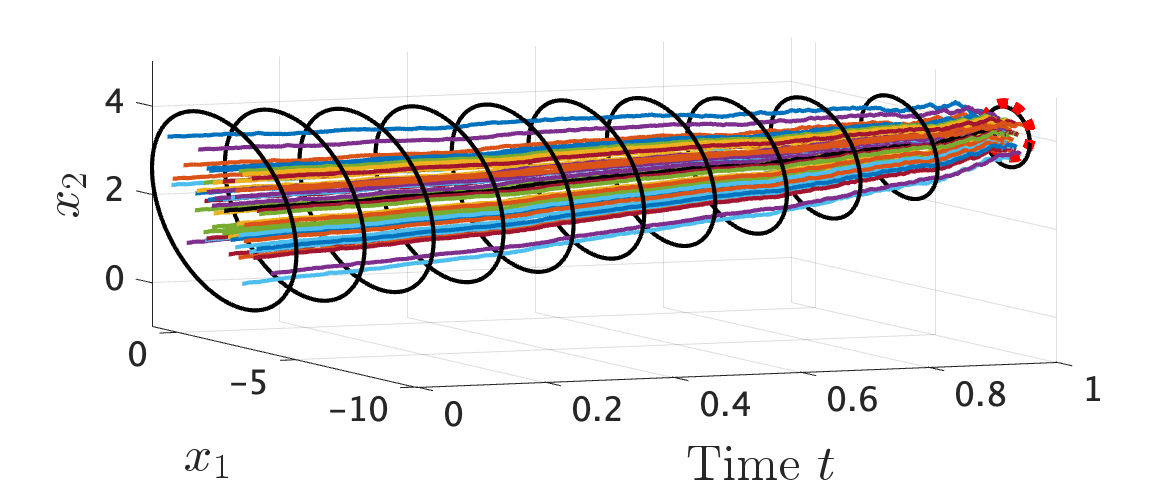}
		\subcaption{$ T_d = 1000 $.}
	\end{minipage}
    \begin{minipage}[b]{1.0\linewidth}
		\centering
		\includegraphics[keepaspectratio, scale=0.43]
		{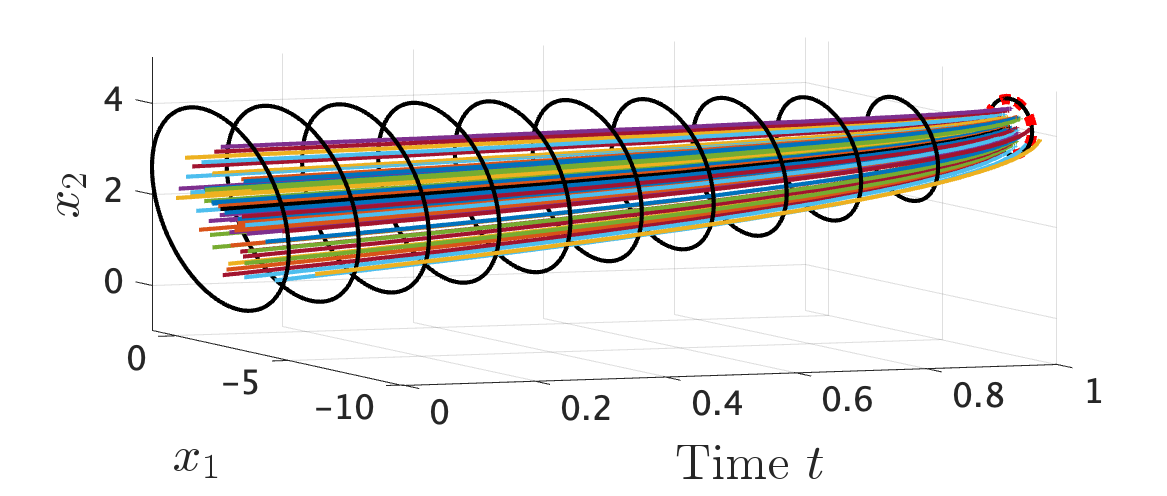}
		\subcaption{$T_d = 100000$.}
	\end{minipage}
	\caption{$ 50 $ samples of $ x(t) = [x_1(t) \ x_2(t) ]^\top $ under the iterative sliced controller \eqref{eq:slice_controller} with different $ T_d = T/h = 100, 1000, 100000 $ (colored lines). The black and red ellipses are the $ 3\sigma $ ellipses of $ \Sigma(t) $ and $ \Sigma_f $, respectively.}\label{fig:iterative}
\end{figure}

Next, Fig.~\ref{fig:compare} compares the iterative sliced controller~\eqref{eq:slice_controller}, the ideal sliced controller~\eqref{eq:optimal_controller_single_integ}, and the minimum energy controller~\eqref{eq:minimum_energy}. It can be seen that the evolution of $ \Sigma(t) $ under the iterative sliced controller with large $ T_d $ is almost identical to that under the ideal sliced controller.
Unlike the minimum energy controller, under which the state paths exhibit straight lines, the sliced controller results in curved trajectories of the state process.

\begin{figure}[t]
	\begin{minipage}[b]{1.0\linewidth}
		\centering
		\includegraphics[keepaspectratio, scale=0.43]
		{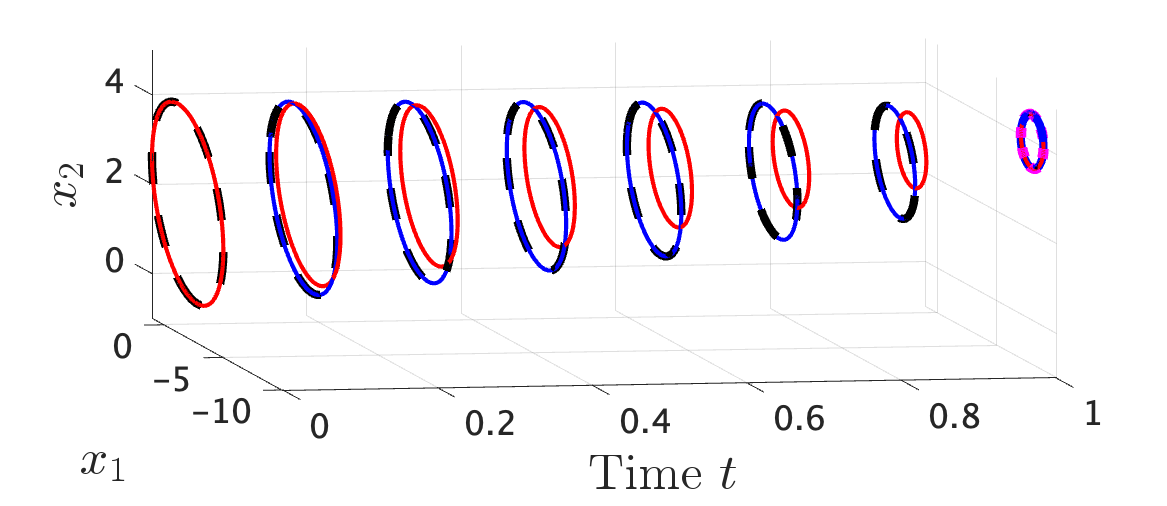}
		\subcaption{$ T_d = 100 $.}\label{fig:compare_density}
	\end{minipage}
    \begin{minipage}[b]{1.0\linewidth}
		\centering
		\includegraphics[keepaspectratio, scale=0.43]
		{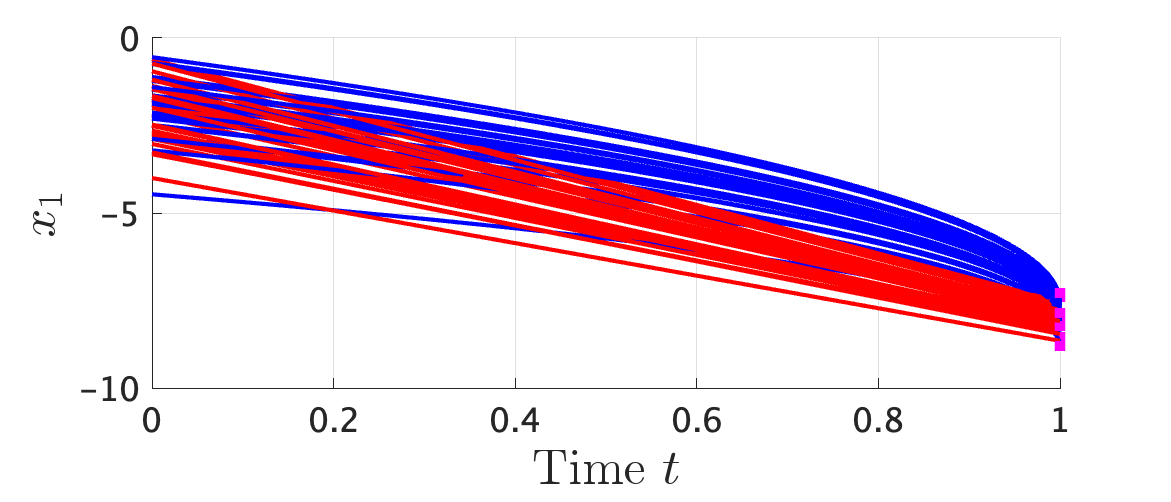}
		\subcaption{$ T_d = 100 $.}\label{fig:compare_path}
	\end{minipage}
	\caption{\subref{fig:compare_density} Evolution of the $ 3\sigma $ ellipse of the covariance $ \Sigma(t) $ under the iterative sliced controller \eqref{eq:slice_controller} with $ T_d = 100000 $ (black, dashed), the ideal sliced controller~\eqref{eq:optimal_controller_single_integ} (blue), and the minimum energy controller~\eqref{eq:minimum_energy} (red). The magenta dashed ellipse shows the $ 3\sigma $ ellipse of $ \Sigma_f $. \subref{fig:compare_path} $ 30 $ samples of $ x_1(t) $ under the ideal sliced controller (in blue) and the minimum energy controller (in red).}\label{fig:compare}
\end{figure}

\section{Concluding Remarks}\label{sec:conclusion}
We developed a sliced optimal transport feedback controller for ensemble steering through one–dimensional projections. In continuous time, the feedback decreases $SW_2^2$ monotonically with gain $(T-t)^{-1}$. For Gaussian marginals, it is affine, preserves Gaussianity, yields closed equations for the mean and covariance, and reaches the target at the terminal time with expected energy equal to one-half of the squared sliced distance. A discrete, direction-by-direction implementation based on one-dimensional maps aligns numerically with the ideal law and produces trajectories that differ qualitatively from the minimum-energy displacement interpolation.

This viewpoint suggests enrich directions while avoiding high-dimensional couplings: (i) output-distribution steering under partial observations, by enforcing families of one-dimensional transports along informative projections so that computation scales with samples and directions; (ii) extensions to general linear plants, posing sliced objectives in the output space and designing feedback consistent with reachability and observability; (iii) uncertainty-aware feedback that reduces not only mean error but also a transport discrepancy to concentrate ensembles in safe regions under disturbances; and (iv) methods beyond Gaussians, including projected--quantile controllers, adaptive or quasi-Monte Carlo direction sampling, and online MPC-alike updates with principled step-size selection. Taken together, operating directly in the sensing/actuation space offers a tractable and interpretable route to distribution control.

\bibliographystyle{ieeepes}
\bibliography{ecc2026_arxiv_slicedOT}

%%%%%%%%%%%%%%%%%%%%%%%%%%%%%%%%%%%%%%%%%%%%%%%%%%%%%%%%%%%%%%%%%%%%%%%%%%%%%%%%
\appendices

\section{Proof of Proposition~\ref{prop:time_derivative}}\label{app:time_derivative}
First, we derive the evolution equation for the sliced density $ \rho_\theta $.
For any test function $ \varphi \in C_c^\infty (\bbR) $, we have
\begin{align}
    \int_{\bbR} \varphi(y) \partial_t \rho_\theta (t,y) \rmd y &= \int_{\bbR^n} \varphi(\theta^\top x) \partial_t \rho(t,x) \rmd x  \nonumber\\
    &= - \int_{\bbR^n} \varphi(\theta^\top x) \nabla_x \cdot (\rho v) \rmd x, \label{eq:test_1}
\end{align}
where we used \eqref{eq:continuity_eq}.
Let $ x_{-i} := \{x_j\}_{j=1}^n \setminus \{x_i\} $.
Then, by the assumption that $ \rho(t,x) v_i (t,x) \rightarrow 0 $ as $ |x_i| \rightarrow \infty $, the right-hand side can be written as
\begin{align}
    &-\sum_{i=1}^n \int_{\bbR^n} \varphi(\theta^\top x) \partial_{x_i} [\rho v_i] \rmd x_i \rmd x_{-i} \nonumber\\
    &= - \sum_i \int_{\bbR^{n-1}}   \biggl( \left[ \varphi(\theta^\top x) \rho v_i  \right]_{x_i = -\infty}^{\infty} \nonumber\\
    &\qquad\qquad - \int_{\bbR} \partial_{x_i} [\varphi(\theta^\top x)] \rho v_i  \rmd x_i \biggr) \rmd x_{-i} \nonumber\\
    &= \int_{\bbR^n} \left(\nabla_x \varphi(\theta^\top x) \right)^\top \rho v \rmd x \nonumber\\
    &= \int_{\bbR^n}  \varphi'(\theta^\top x)  \theta^\top v \rho \rmd x = \int_{\bbR} \varphi'(y) \Phi_\theta( t,y) \rmd y , \label{eq:test_2}
\end{align}
where $ \varphi' $ is the derivative of $ \varphi $, $ \Phi_\theta (t,y) := \int_{\bbR^n} \theta^\top v(t,x) \rho(t,x) \delta(y - \theta^\top x) \rmd x $, and $ \delta $ denotes the Dirac delta function.

By integrating by parts and noting that $ \varphi  $ has a compact support, it follows from \eqref{eq:test_1} and \eqref{eq:test_2} that
\begin{equation}
    \int_{\bbR} \varphi(y) \partial_t \rho_\theta (t,y) \rmd y = - \int_{\bbR} \varphi(y) \partial_y \Phi_\theta(t,y) \rmd y .
\end{equation}
Since the above equation holds for any $ \varphi \in C_c^\infty (\bbR) $, by \cite[Remark~3.7]{Duistermaat2010}, the following equation holds for all $ t \in (0,T) $ and almost all $ y \in \bbR $:
\begin{equation}\label{eq:sliced_continuity_1}
    \partial_t \rho_\theta (t,y) = - \partial_y \Phi_\theta (t,y)  .
\end{equation}
Now, define
\begin{align*}
    v_\theta (t,y) := \frac{\Phi_\theta (t,y)}{\rho_\theta (t,y)} = \bbE\! \left[\theta^\top v(t,x(t)) \mid \theta^\top x(t) = y \right] ,  \nonumber & \\
    (t,y) \in [0,T] \times \bbR .&
\end{align*}
Then, \eqref{eq:sliced_continuity_1} can be written as
\begin{equation}\label{eq:slice_continuity}
    \partial_t \rho_\theta (t,y) = - \partial_y \! \left[ \rho_\theta (t,y) v_\theta (t,y) \right] ,
\end{equation}
which is the continuity equation for $ \rho_\theta $ under the velocity field $ v_\theta $ along the direction $ \theta $.

By \cite[Theorem~8.4.7]{Ambrosio2005} and the sliced continuity equation~\eqref{eq:slice_continuity}, for almost all $ t \in (0,T) $,
\begin{align}
    \partial_t W_2^2 (\rho_\theta (t,\cdot), \rho_{f,\theta}) &= 2 \int_{\bbR} \left(y_1 - y_2 \right) v_\theta (t,y_1)  \gamma^* (\rmd y_1, \rmd y_2), \nonumber
\end{align}
where $ \gamma^* $ is an optimal coupling that minimizes the right-hand side of \eqref{def:wasserstein} with $ \mu(\rmd y_1) = \rho_\theta(t,y_1) \rmd y_1 $ and $ \nu(\rmd y_1) = \rho_{f,\theta} (y_2) \rmd y_2 $.
Moreover, by \cite[Theorem~2.44]{Villani2003}, $ \gamma^* $ takes the form $ \gamma^* (\rmd y_1, \rmd y_2) = \rho_\theta(t,y_1) \delta(y_2 - \calT_t^\theta (y_1)) \rmd y_1 \rmd y_2 $, where $ \calT_t^\theta : \bbR \rightarrow \bbR $ is the unique optimal transport map from $ \rho_\theta (t,\cdot) $ to $ \rho_{f,\theta} $. Therefore, we obtain
\begin{align}
    \partial_t W_2^2 (\rho_\theta (t,\cdot), \rho_{f,\theta}) &= 2 \int_{\bbR} \left(y - \calT_t^\theta (y) \right) v_\theta (t,y)  \rho_\theta (t,y) \rmd y \nonumber .
\end{align}
By combining this with the expression:
\begin{equation}
    \rho_\theta (t,y) = \int_{\bbR^n} \delta (y - \theta^\top x) \rho(t,x) \rmd x , ~~ (t,y) \in \bbR_{\ge 0} \times \bbR , \nonumber
\end{equation} 
we have that for almost all $ t \in (0,T) $,
\begin{align}
    &\partial_t W_2^2 (\rho_\theta (t,\cdot), \rho_{f,\theta}) \nonumber\\
    &= 2 \int_{\bbR^n} \left(\theta^\top x - \calT_t^\theta (\theta^\top x) \right) \theta^\top v(t,x) \rho(t,x) \rmd x . \label{eq:W2_dt}
\end{align}
By substituting this into \eqref{eq:interchange}, we obtain the desired result.

\section{Proof of Proposition~\ref{prop:linearlity_Gaussian}}\label{app:linearity_Gaussian}
Let $ K^*(t) := K(t,\Sigma(t)) $ and $ \eta^*(t) := \eta(t,\Sigma(t)) $. Then, we show that the solution $ \wtilde{x} $ of $ \dot{\wtilde{x}}(t) = K^*(t) \wtilde{x} (t) + \eta^*(t) $ with $ \wtilde{x}(0) = x(0) $ coincides with that of $ \dot{x} (t) = v(t,x(t)) $, where $ v $ is given by \eqref{eq:optimal_controller_single_integ}.
It is well-known that under the linearity of the system of $ \wtilde{x} $ and the Gaussianity of $ \wtilde{x}(0) $, the solution $ \wtilde{x}(t) $ follows $ \calN(\wtilde{m}(t),\Sigma(t)) $ for any $ t \in [0,T) $, where $ \wtilde{m}(t):= \bbE[\wtilde{x}(t)] $.
In addition, we have $ \dot{\wtilde{m}}(t) = K^*(t) \wtilde{m}(t) + \eta^*(t) $, $ \wtilde{m} (0) = m_0 $, and
\begin{align}
    \dot{m} (t) &= - \frac{\lambda(t)}{n} (m(t) - m_f) = K^* (t) m(t) + \eta^* (t) , \label{eq:m_ode} \\
    m(0) &= m_0 . \nonumber
\end{align}
Since \eqref{eq:m_ode} admits a unique solution on $ [0,T) $, we have $ \wtilde{m}(t) = m(t) $ for any $ t \in [0,T) $.

Denote by $ \wtilde{\rho}_\theta (t,\cdot) $ the sliced density function of $ \wtilde{x}(t) \sim \calN(m(t),\Sigma(t)) $ along the direction $ \theta $, and let $ \wtilde{\calT}_t^\theta  $ be the unique optimal transport map from $ \wtilde{\rho}_{\theta} (t,\cdot) $ to $ \rho_{f,\theta} $.
Then, by the same argument as for \eqref{eq:initial_feedback}, we have that for any $ (t,x) \in [0,T) \times \bbR^n $,
\begin{align}
\wtilde{v}(t,x)  &:= - \lambda(t) \int_{\calS^{n-1}} \left(\theta^\top x - \wtilde{\calT}_t^\theta (\theta^\top x) \right) \theta \sigma (\rmd \theta) \nonumber\\
&= K^*(t) x + \eta^* (t) . \nonumber
\end{align}

Therefore, the solution $ \wtilde{x} $ of $ \dot{\wtilde{x}} (t) = K^*(t) \wtilde{x} (t) + \eta^* (t) = \wtilde{v} (t,\wtilde{x}(t)) $ with $ \wtilde{x}(0) = x(0) $ also solves $ \dot{x} (t) = v(t,x(t)) $. By the uniqueness assumption of the solution $ x $, we obtain the desired result.

\section{Proof of Theorem~\ref{thm:convergence}}\label{app:convergence}
Let $ \bar{K}(t) := K(t,\Sigma(t))/\lambda(t) $ and $ \bar{\eta}(t) := \eta(t,\Sigma(t)) / \lambda(t) $.
Note that by the linearity of $ v(t,x) $ in $ x $ and the Gaussianity of $ \rho(t,x) $ (see Proposition~\ref{prop:linearlity_Gaussian}), the assumption in Proposition~\ref{prop:time_derivative} that $ \lim_{|x_i|\rightarrow \infty} \rho (t,x) v_i (t,x) = 0 $ is satisfied.
Then, by Proposition~\ref{prop:time_derivative}, we have
\begin{align}
    &\frac{\rmd }{\rmd t} SW_2 (\rho(t,\cdot),\rho_f)^2 \nonumber\\
    &= -2\lambda(t) \int_{\bbR^n} \left\| \bar{K}(t) x  + \bar{\eta}(t) \right\|^2 \rho(t,x) \rmd x \nonumber\\
    &= -2\lambda(t) \bbE\! \left[ \| \bar{K}(t) x(t) + \bar{\eta}(t) \|^2 \right] . \nonumber
\end{align}
By the change of variables $ \tau = -\log (T-t) $, we obtain
\begin{align}
    \frac{\rmd }{\rmd \tau} SW_2(\rho(t,\cdot),\rho_f)^2 &= \frac{\rmd t}{\rmd \tau} \frac{\rmd }{\rmd t} SW_2 (\rho(t,\cdot),\rho_f)^2 \nonumber\\
    &= (T-t) \frac{\rmd }{\rmd t} SW_2(\rho(t,\cdot),\rho_f)^2 \nonumber\\
    &= -2 \bbE \! \left[ \| \bar{K}(t) x(t) + \bar{\eta}(t) \|^2 \right] .  \label{eq:dSW_dtau} 
\end{align}
In addition, we have
\begin{align}
    &\bbE\! \left[ \| \bar{K}(t) x(t) + \bar{\eta}(t) \|^2 \right] \nonumber\\
    &=\bbE\! \left[ \Big\| \bar{K}(t) (x(t) - m(t)) - \frac{1}{n} (m(t) - m_f) \Big\|^2 \right] \nonumber\\
    &= \bbE\! \left[ \left\| \bar{K}(t) (x(t) - m(t)) \right\|^2 \right] + \frac{1}{n^2} \left\| m(t) - m_f \right\|^2  , \label{eq:decompose}
\end{align}
where we used $ \bbE[x(t) - m(t)] = 0 $.
Since $ SW_2 (\rho(t,\cdot),\rho_f)^2 $ is lower bounded and $ \frac{\rmd }{\rmd \tau} SW_2 (\rho(t,\cdot),\rho_f)^2 \le 0 $ for almost all $ \tau \in (-\log T,\infty) $, it holds that 
\[
\lim_{\tau \rightarrow \infty} \frac{\rmd }{\rmd \tau} SW_2^2(\rho(t,\cdot),\rho_f) = 0 .
\]
By combining this with \eqref{eq:dSW_dtau} and \eqref{eq:decompose}, we obtain 
\begin{align}
    \lim_{\tau \rightarrow \infty} \bbE\! \left[ \left\| \bar{K}(t) (x(t) - m(t)) \right\|^2 \right] &= 0 , \label{eq:converge_zero1}\\
    \lim_{\tau \rightarrow \infty}  \left\| m(t) - m_f \right\|^2 &= 0,
\end{align}
where the second equation means $ \lim_{t\nearrow T} m(t) = m_f $.

For the left-hand side of \eqref{eq:converge_zero1}, we have
\begin{align}
    \bbE\! \left[ \left\| \bar{K}(t) (x(t) - m(t)) \right\|^2 \right] &= \trace \! \left( \bar{K} (t) \Sigma(t) \bar{K} (t)^\top \right) \nonumber\\
    &\ge \lambda_{\min} (\Sigma(t)) \| \bar{K} (t) \|_{\rm F}^2 ,\nonumber
\end{align}
where $ \| \cdot \|_{\rm F} $ denotes the Frobenius norm and the minimum eigenvalue $ \lambda_{\min} (\Sigma(t)) $ of $ \Sigma(t) $ is positive for any $ t \in [0,T] $ by assumption. Thus, it follows from \eqref{eq:converge_zero1} that $ \lim_{t\nearrow T} \| \bar{K} (t) \|_{\rm F}^2 = 0 $. In other words,
\begin{align*}
\lim_{t\nearrow T} \int_{\calS^{n-1}} \sqrt{\frac{\theta^\top \Sigma_f \theta}{\theta^\top \Sigma(t) \theta}} \theta \theta^\top \sigma(\rmd \theta) - \frac{1}{n} I = 0.
\end{align*}
Lastly, by Lemma~\ref{lem:harmonic} in Appendix~\ref{app:harmonic}, we obtain $ \lim_{t\nearrow T} \Sigma(t) = \Sigma_f $, which completes the proof.

\section{Proof of Theorem~\ref{thm:energy}}\label{app:energy}
Let $ \bar{K}(t) := K(t,\Sigma(t))/\lambda(t) $ and $ \bar{\eta}(t) := \eta(t,\Sigma(t)) / \lambda(t) $.
Then, we have
\begin{align*}
    \bbE\! \left[\int_0^T \| u(t) \|^2 \rmd t \right] &= \int_0^T \lambda(t)^2 \bbE\! \left[ \| \bar{K}(t) x(t) + \bar{\eta}(t) \|^2 \right] \rmd t \\
    &= -\frac{1}{2} \int_0^T \frac{1}{T-t} \frac{\rmd }{\rmd t} SW_2 (\rho(t,\cdot), \rho_f)^2 \rmd t.
\end{align*}
Let $ f(t) := SW_2 (\rho(t,\cdot), \rho_f)^2 $. By the change of variables $ \tau = -\log (T-t) $, we get
\begin{align*}
    \bbE\! \left[\int_0^T \| u(t) \|^2 \rmd t \right] &= -\frac{1}{2} \int_{\tau(0)}^\infty \frac{1}{T-t}  \frac{\rmd f(t(\tau))}{\rmd \tau} (T-t) \rmd \tau \\
    &= -\frac{1}{2} \int_{\tau(0)}^\infty \frac{\rmd f(T-\Ee^{-\tau})}{\rmd \tau} \rmd \tau \\
    &= \frac{1}{2}\left(f(0) - \lim_{t\nearrow T} f(t) \right).
\end{align*}
By Theorem~\ref{thm:convergence}, we have $ \lim_{t\nearrow T} f(t) = 0 $. Therefore, we obtain the desired result.

\section{}\label{app:harmonic}
\begin{lemma}\label{lem:harmonic}
Let $ \Sigma \succ 0 $, $\Sigma_f \succeq 0$, $ m \in \bbR^n $, and assume that the following equation holds.
\begin{align}\label{eq:coeff1_zero2}
    \left(\int_{\calS^{n-1}} \sqrt{\frac{\theta^\top \Sigma_f \theta}{\theta^\top \Sigma \theta}} \theta \theta^\top \sigma(\rmd \theta) \right) - \frac{1}{n} I &= 0 . 
\end{align}
Then, we have $ \Sigma = \Sigma_f $.
\end{lemma}

\begin{proof} 
Define $ d(\theta) := \theta^\top \Sigma \theta $ and $ g(\theta) := \theta^\top \Sigma_f \theta $.
Multiplying both sides of \eqref{eq:coeff1_zero2} by $ \Sigma $ and taking the trace, we have
\begin{equation}
    \int_{\calS^{n-1}} \sqrt{d(\theta) g(\theta)} \sigma(\rmd \theta) = \int_{\calS^{n-1}} d(\theta)  \sigma(\rmd \theta) , \label{eq:g_d_cond1}
\end{equation}
where we used $ \int_{\calS^{n-1}} d(\theta) \rmd \sigma = \trace(\Sigma) / n $. Similarly, multiplying both sides of \eqref{eq:coeff1_zero2} by $ \Sigma_f $, we obtain
\begin{equation}
    \int_{\calS^{n-1}} \frac{g(\theta)^{3/2}}{d(\theta)^{1/2}} \sigma(\rmd \theta)  = \int_{\calS^{n-1}} g(\theta)  \sigma(\rmd \theta) .\label{eq:g_d_cond2}
\end{equation}

Let us consider the following integral:
\begin{align*}
&\int_{\calS^{n-1}} \frac{(\sqrt{g} - \sqrt{d})^2 (\sqrt{g} + \sqrt{d})}{\sqrt{d}} \rmd \sigma \nonumber\\
&= \int \frac{g^{3/2}}{d^{1/2}} \rmd \sigma - \int g \rmd \sigma - \int \sqrt{dg} \rmd \sigma + \int d \rmd \sigma .
\end{align*}
By \eqref{eq:g_d_cond1} and \eqref{eq:g_d_cond2}, the above integral equals zero. Moreover, since $ (\sqrt{g} + \sqrt{d}) / \sqrt{d} > 0 $, it follows that $ d(\theta) = g(\theta) $ for any $ \theta \in \calS^{n-1} $, which means $ \Sigma = \Sigma_f $.
\end{proof}

\end{document}